\documentclass[12pt,epsfig,amsfonts]{amsart}

\usepackage{amsmath,amssymb, amsthm, epsfig}

\newtheorem{prop}{Proposition}[section]
\newtheorem{definition}{Definition}[section]
\newtheorem{cor}{Corollary}[section]
\newtheorem{theorem}{Theorem}[section]
\newtheorem*{theorema}{Theorem A}
\newtheorem*{theoremb}{Theorem B}

\newtheorem{lemma}{Lemma}[section]

\newtheorem{sublemma}{Sublemma}[section]
\newtheorem{remark}{Remark}[section]

\begin{document}

\author{Hiroki Takahasi *}
\thanks{* Department of Mathematics, Faculty of Science,
Kyoto University, Kyoto 606-8502, JAPAN\\
(e-mail address: takahasi@math.kyoto-u.ac.jp)}


\title{Statistical properties of nonuniformly expanding
1d maps with logarithmic singularities}

\begin{abstract}
For a certain parametrized family of maps on the circle, with
critical points and logarithmic singularities where derivatives
blow up to infinity, a positive measure set of parameters was
constructed in \cite{TW09}, corresponding to maps which exhibit
nonuniformly hyperbolic behavior. For these parameters, we prove
the existence of absolutely continuous invariant measures with
good statistical properties, such as exponential decay of
correlations. Combining our construction with the logarithmic
nature of the singularities, we obtain a positive variance in
Central Limit Theorem,
for any nonconstant H\"older continuous observable. 
\end{abstract} \maketitle

\section{Introduction}
Let $f_{a, L}: {\mathbb R} \to {\mathbb R}$ be such that
\begin{equation}\label{f1-s1}
f_{a,L}\colon x\mapsto x+a+L\ln |\Phi(x)|,
\end{equation}
where $a \in [0,1]$, $L \in {\mathbb R}$ are real parameters and
$\Phi(x)$ is such that $\Phi(x+1) = \Phi(x)$. We assume that
$\Phi(x)$ is a Morse function, the graph of $y = \Phi(x)$
intersects $x$-axis, and all the intersections are transverse. The
functions $f_{a, L}$ induce a two parameter family of
endomorphisms on $S^1={\mathbb R} /{\mathbb Z}$, having
non-degenerate critical points and singularities where the value
of $f_{a,L}$ is undefined. For sufficiently large $|L|$, a
positive measure set $\Delta(L)$ of the parameter $a$ was
constructed in \cite{TW09}, such that if $a \in \Delta(L)$, then
$f_{a, L}$ admits an invariant measure that is absolutely
continuous with respect to Lebesgue measure (acim). In this paper
we study statistical properties of this measure.

This class of systems is motivated by the recent studies
\cite{Wa,WO,WOk} on homoclinic tangles and strange attractors in
periodically forced differential equations ($S^1$ reflects the
time-periodicity of the force). In brief terms, the maps $f_{a,L}$
as we treat here can be obtained by considering first-return maps
of the flow (in the extended phase space introducing the time as a
new variable) to appropriate cross-sections, and then passing to a
singular limit. This last step results in a considerable
simplification of the dynamics. Nevertheless, the map $f_{a,L}$
retains a large share of the complexity of the corresponding flow,
and thus, provide an important insight to its behavior. 

Apart from this original motivation, the family of circle maps is
of interest in its own light, for the feature of the logarithmic
singularities that turns out to influence on some statistical
properties of the acips, as we explain in the sequel.

\subsection{Statements of the results} For smooth maps on the interval
or the circle, it is now classical that an exponential growth of
derivatives along the orbits of critical points implies the
existence of acims with good statistical properties
\cite{BV,BLS,KN,Y1,Y2}. Our first result is a version of this for
$f_{a,L}$ with critical and singular points. Dynamical properties
shared by maps corresponding to parameters in $\Delta(L)$ are
listed in Section \ref{gs}.

\begin{theorema}
For any $f\in\{f_{a,L}\colon a\in\Delta(L)\}$ there exists an
ergodic $f$-invariant probability measure $\mu$ that is equivalent
to Lebesgue measure. In addition,
\smallskip

\noindent(1) for any $\eta\in(0,1]$ there exists $\tau\in(0,1)$
such that for any H\"older continuous function $\varphi$ on $S^1$
with H\"older exponent $\eta$ and $\psi\in L^\infty(\mu)$, there
exists a constant $K(\varphi,\psi)$ such that
$$\left| \int(\varphi\circ f^n)\psi d\mu-\int\varphi
d\mu\int\psi d\mu\right|\leq K(\varphi,\psi)\tau^{n}\quad\text{for
every }n>0;$$
\smallskip

\noindent (2) $(f,\mu)$ satisfies Central Limit Theorem (see the
definition below).
\end{theorema}

This is not the first result on statistical properties of
one-dimensional maps with critical and singular points. A certain
class of maps were studied in \cite{DHL06}, including the
Lorenz-like maps corresponding to positive measure sets of
parameters constructed in \cite{LT,LV}. Maps with singularities
and infinitely many critical points were studied in \cite{PRV}. To
our knowledge, however, there is no previous study on statistical
properties of maps with logarithmic singularities. For instance,
one key aspect of our maps that has no analogue in those of the
previous studies is that, returns to a neighborhood of
singularities can happen very frequently. The previous arguments
seem not sufficient to deal with points like this.

In the study of dynamical systems with singularities, influences
of singularities on dynamics are not well understood. Indeed,
singularities with blowing up derivatives help to create
expansion, and to enforce a chaotic behavior. However, little is
known on influences of singularities on statistical properties of
the systems. In this direction, one result we are aware is
\cite{LMP04} which takes advantage of the singularity of the
expanding Lorenz map to show that the Lorenz attractor is mixing.
In the proof of Theorem A, we design our construction in such a
way that allows us to draw a new conclusion on Central Limit
Theorem, viewed as an influence of the logarithmic singularities.

Let $g\colon X\to X$ be a dynamical system preserving a
probability measure $\nu$. We say $(g,\nu)$ satisfies Central
Limit Theorem if for any H\"older continuous function $\phi$ on
$X$ with $\int\phi d\nu=0$,
$$ \frac{1}{\sqrt{n}}\sum_{i=0}^{n-1}\phi\circ g^i \
\longrightarrow \ \mathcal N(0,\sigma)\quad\text{ in
distribution},$$ where $\mathcal N(0,\sigma)$ is the normal
distribution with mean $0$ and variance $\sigma^2$, and
$$\sigma^2=\lim_{n\to\infty}\frac{1}{n}
\int\left(\sum_{i=0}^{n-1}\phi\circ g^i\right)^2d\nu.$$
 If $\sigma>0$,
this means that for every interval $J\subset\mathbb R,$
$$\nu\left\{x\in X\colon
\frac{1}{\sqrt{n}}\sum_{i=0}^{n-1}\phi(g^i(x))\in J\right\} \
\longrightarrow \
\frac{1}{\sigma\sqrt{2\pi}}\int_Je^{-\frac{t^2}{2\sigma^2}}dt.$$
It is known (see e.g. \cite{PP,Y2}) that $\sigma>0$ if and only if
$\phi$ is {\it not coboundary}, that is, the cohomological
equation
\begin{equation*}
\phi=\psi\circ g-\psi\end{equation*} has no solution in
$L^2(\nu)$. Otherwise, $\phi$ is called {\it coboundary}. For
dynamical systems satisfying Central Limit Theorem, determining
the largest possible classes of functions which are not coboundary
is an intricate problem, even for Axiom A systems \cite{PP}. For
countable Markov maps on intervals, Morita \cite{Mo} obtained
Central Limit Theorem for a broad class of functions including
those with bounded variation, and proved that there exists no
non-trivial function which is coboundary. Our construction in
Theorem A and the nature of the singularities allow us to show
that, for our maps, there exists no non-trivial H\"older
continuous function which is coboundary.
\begin{theoremb}
Let $(f,\mu)$ be as above. If a H\"older continuous $\phi\colon
S^1\to\mathbb R$ with $\int\phi d\mu=0$ is coboundary, then
$\phi\equiv 0$.
\end{theoremb}

Our strategy in Theorem A is to construct an induced Markov map
and apply the scheme of Young \cite{Y2}. A key feature of this
construction is that the domain of the induced Markov map is a
full measure subset of $S^1$. In other words, $S^1$ is cut into
pieces, and each piece grows to the entire $S^1$ in a controlled
way. As a consequence, $\mu$ is equivalent to Lebesgue.

A proof of Theorem B is outlined as follows. Suppose that $\phi$
is coboundary, with an $L^2$ solution $\psi$. Using the induced
Markov map, it is possible to show that $\psi$ has a version
$\tilde\psi$ (i.e. $\psi=\tilde\psi$ $\mu$-a.e.) which is
(H\"older) continuous on the entire $S^1$. On the other hand, the
distinctive property of the logarithmic singularities is that, a
small neighborhood of a singular point is divided into a countable
number of intervals, and each of them is sent to the entire $S^1$
just by one iterate. This property allows us to rule out the
existence of nonconstant continuous solution of the cohomological
equation. Hence, $\tilde\psi$ has to be a constant function, and
$\phi\equiv0$ follows.

The rest of this paper consists of four sections. In Section 2 we
collect necessary materials in \cite{TW09} as far as we need them.
In Section 3 we perform a large deviation argument, a key step for
the construction of the induced Markov map. In Section 4 we put
these results together and construct an induced Markov map with
exponential tails, and prove the theorems. In Section 5 we prove
an entropy formula, connecting the metric entropy to the Lyapunov
exponent.
\section{Properties of nonuniformly expanding maps}
This section collects materials in \cite{TW09} as far as we need
them. Dynamical properties shared by maps corresponding to the
parameters in $\Delta(L)$ are stated in Section \ref{gs}.










\subsection{Elementary facts}\label{s2.1}
From this point on we use $L$ for both $L$ and $|L|$. We take $L$
as a base of the logarithmal function ${\rm log}(\cdot)$. For $f =
f_{a, L}$, let $C(f) = \{ f'(x) = 0 \}$ denote the set of critical
points and $S(f) = \{ \Phi(x) = 0 \}$ the set of singular points.
The distances from $x\in S^1$ to $C(f)$ and $S(f)$ are denoted by
$d_C(x)$ and $d_S(x)$ respectively. For $\varepsilon > 0$, we use
$C_{\varepsilon}$ and $S_\varepsilon$ to denote the
$\varepsilon$-neighborhoods of $C$ and $S$ respectively.

\begin{lemma}\label{derivative}
{\rm [\cite{TW09} Lemma 1.1.]} There exists $K_0>1$ and
$\varepsilon_0 > 0$, such that for all $L$ sufficiently large and
$f=f_{a,L}$,
\smallskip

\noindent{\rm (a)} for all $x \in S^1$,
$$ K_0^{-1}L
\frac{d_C(x)}{d_S(x)} \leq|f'x| \leq K_0L \frac{d_C(x)}{d_S(x)}, \
\ \ \ \ \  |f''x| \leq \frac{K_0L}{d^2_S(x)};
$$

\noindent{\rm (b)} for all $\varepsilon
>0$ and $x\not \in C_{\varepsilon}$,
$ |f'x|\geq K_0^{-1} L\varepsilon$; and

\noindent{\rm (c)} for all $x\in C_{\varepsilon_0}$, $K^{-1}_0 L <
|f''x| < K_0 L$.
\end{lemma}
\noindent{\it Sketch of the proof.} Use the assumptions on
$\Phi(x)$: $\Phi'(x) \neq 0$ on $\{ \Phi(x) = 0 \}$; $\Phi''(x)
\neq 0$ on $\{ \Phi'(x) = 0 \}$. \qed
\subsection{Bounded distortion}
For $x\in S^1$, $n\geq1$, let
\begin{equation}\label{Theta}
D_n(x)=\frac{1}{\sqrt{L}}\cdot\left[\sum_{0\leq i<n}
d_i^{-1}(x)\right]^{-1} \ \ \ \ \text{where} \ \ \ \ \
d_i(x)=\frac{d_C(f^ix)\cdot d_S(f^ix)}{|(f^i)'x|},
\end{equation}
when they make sense.

\begin{lemma}\label{dist}
The following holds for all sufficiently large $L$: if $n\geq1$
and $x,fx,\cdots,f^{n-1}x\notin C\cup S$, then for all
$\xi,\eta\in[x-D_n(x),x+D_n(x)]$,
$$\frac{|(f^n)'\xi|}{|(f^n)'\eta|}\leq 2\text{ and }
\left|\frac{|(f^n)'\xi|}{|(f^n)'\eta|}-1\right|\leq
\frac{1}{L^{1/3}}\frac{|f^n\xi-f^n\eta|} {D_n(x)|(f^n)'x|}.$$
\end{lemma}

\begin{remark}\label{south}
{\rm In Section 3 we will use these estimates on a bigger
 interval (comparable in length), but this does not seriously
affect the estimates.}
\end{remark}

\begin{proof}
The first estimate was in [\cite{TW09} Lemma 1.1]. We prove the
second one. Let $I$ denote the subinterval of
$[x-D_n(x),x+D_n(x)]$ with endpoints $\xi,\eta$. Let $i\in[0,n)$.
By Lemma \ref{derivative}, for any $\phi\in f^iI,$
$$\frac{|f''\phi|}{|f'\phi|}
\leq\frac{ K_0^2}{d_C(\phi)d_S(\phi)}\leq\frac{ 2K_0^2}
{d_C(f^ix)d_S(f^ix)},$$ where the last inequality follows from
$$|f^iI|\leq 2D_n(x)|(f^i)'x|\leq\frac{2}{\sqrt{L}} d_C(f^ix)\cdot
d_S(f^ix)\ll\max\{d_C(f^ix), d_S(f^ix)\}.$$ We also have
$$|f^iI|\leq
2|(f^i)'x||\xi-\eta|d_i(x)d_i^{-1}(x)= 2|\xi-\eta|d_C(f^ix)\cdot
d_S(f^ix)d_i^{-1}(x).$$ Multiplying these two inequalities,
\begin{equation*}
|f^iI|\sup_{\phi\in f^iI}\frac{|f''\phi|}{|f'\phi|}\leq
4K_0^2|\xi-\eta| d_i^{-1}(x).\end{equation*} Summing this over all
$0\leq i<n$ we obtain
\begin{align*}
{\rm ln} \frac{|(f^n)'\xi|}{|(f^n)'\eta|}&\leq\sum_{0\leq i<n}
{\rm ln} \frac{|f'(f^{i}\xi)|}{|f'(f^{i}\eta)|} \leq\sum_{0\leq
i<n}|f^iI|\sup_{\phi\in
f^iI}\frac{|f''\phi|}{|f'\phi|}\\
&\leq \frac{4K_0^2|\xi-\eta|}{\sqrt{L}D_n(x)} \leq \frac{8K_0^2
|f^n\xi-f^n\eta|}{\sqrt{L}D_n(x)|(f^n)'x|}\leq
\frac{1}{L^{1/3}}\frac{|f^n\xi-f^n\eta|} {D_n(x)|(f^n)'x|}.
\end{align*}
The desired inequality holds.
\end{proof}

\subsection{Uniform expansion outside of critical regions}\label{good}
Let $\lambda = 10^{-3}$, $\alpha = 10^{-6}$ and $\delta =
L^{-\alpha N_0}$, where $N_0$ is a large integer. Let $\sigma
=L^{-\frac16}$. For $c \in C(f)$, let $v_0 = f(c)$ and $\{ v_i=
f^i v_0, \ i \in {\mathbb Z}^+ \}$.

\begin{lemma}{\rm [\cite{TW09} Lemma 1.3.]}\label{outside}
There exists a large integer $N_0$ such that the following holds
for all sufficiently large $L$: assume for each $c\in C(f)$ and
every $0\leq n\leq N_0,$ $d_C(v_n)\geq\sigma$ and
$d_S(v_n)\geq\sigma$, then:
\smallskip

\noindent{\rm (a)} if $n\geq1$ and $x$, $fx,\cdots, f^{n-1}x\notin
C_{\delta}$, then $|(f^{n})'x|\geq \delta L^{2\lambda n}$;

\noindent{\rm (b)} if moreover $f^nx\in C_{\delta}$, then
$|(f^{n})'x| \geq L^{2 \lambda n}$.
\end{lemma}
\noindent{\it Sketch of the proof.} Let
$\delta_0=L^{-\frac{11}{12}}\gg\delta$. By Lemma \ref{derivative},
derivatives grow exponentially, as long as orbits stay outside of
$C_{\delta_0}$. Once they fall in $C_{\delta_0}\setminus
C_\delta$, they copy the growth of the derivatives of the nearest
critical orbit for a certain period of time. The choice of
$\delta$ and the assumption on $C(f)$ together ensure that this
period is enough to recover an exponential growth. \qed
\medskip

\subsection{Dynamical assumptions}\label{gs}

For the rest of this paper, we assume that $N_0,L$ are large so
that the conclusions of the previous three lemmas hold. In
addition, for each $c \in C(f)$ we assume:
\smallskip

\noindent{\rm (a)} for $0 \leq n \leq N_0$, $d_C(v_n)>\sigma, \
d_S(v_n)
> \sigma$;

\noindent{\rm (b)} for every $n > N_0$,

\ \ \ {$(G1)$} $|(f^{j-i})'v_i|\geq L \min\{\sigma, L^{-\alpha
i}\}\text{ for every }0\leq i<j\leq n+1;$

\ \ \ {$(G2)$}  $|(f^i)'v_0|\geq L^{\lambda i}$ for every $0<i\leq
n+1$;

\ \ \ {$(G3)$} $d_S(v_i)\geq L^{-4\alpha i}$ for every $N_0\leq
i\leq n$.
\smallskip

Building on these standing assumptions, we construct an induced
Markov map and deduce the properties in the theorems. It was
proved in \cite{TW09} that there exist a large integer $N_0$ and
$L_0>0$ such that for all $L\geq L_0$, there exists a set
$\Delta(L) \subset [0,1)$ of the parameter $a$ with positive
Lebesgue measure, such that (a) (b) hold for all $f\in\{f_{a,
L}\colon a\in\Delta(L)\}$. In addition,
$\lim_{L\to\infty}|\Delta(L)|\to1$ holds.

\subsection{Recovering expansion}\label{recovery}
Let us introduce bound periods and recovery estimates from small
derivatives near the critical set. Let $c\in C$ and $v_0 = f(c)$.
For $p\geq2$, let
$$
I_{p}(c)=\left(c+ \sqrt{D_{p}(v_0)/(K_0L)},c+ \sqrt{D_{
p-1}(v_0)/(K_0L)}\right].
$$
Let $I_{-p}(c)$ be the mirror image of $I_{p}(c)$ with respect to
$c$.

If $x\in I_p(c) \cup I_{-p}(c)$, then $ |f x - v_0| \leq
D_{p-1}(v_0)$ holds. According to Lemma \ref{dist}, the
derivatives along the orbit of $fx$ shadow that of the orbit of
$v_0$ for $p-1$ iterates. We regard the orbit of $x$ as bound to
the orbit of $c$ up to time $p$, and call $p$ the {\it bound
period} of $x$ to $c$.

\begin{lemma}\label{reclem1}{\rm [\cite{TW09} Lemma 1.6.]}
 For every $p\geq 2$ and $x\in
I_{p}(c)\cup I_{- p}(c)$,
\smallskip

\noindent{\rm (a)} $p\leq\log |c-x|^{-\frac{2}{\lambda}}$;

\noindent{\rm (b)} if $x\in C_\delta$, then $|(f^{p})'x|\geq
\max\left\{|c-x|^{-1+\frac{16\alpha}{\lambda}},L^{\frac{\lambda}{3}p}\right\}$.
\end{lemma}
\noindent{\it Sketch of the proof.} (a) follows from the
definition of $D_p(v_0)$ and the assumption (G2) on $v_0$. The
bounded distortion of $f^{p-1}$ on $f(I_p(c)\cup I_{-p}(c))$ and
(G1), (G3) are used to prove (b). \qed

\subsection{Decomposition into bound/free segments}\label{decomp}
We introduce a useful language along the way. Let $x\in
S^1\setminus (C\cup S)$. Let $$0\leq n_1<n_1+p_1 \leq n_2<n_2+p_2
\leq \cdots$$ be defined as follows: $n_1$ is the smallest $j\geq
0$ such that $f^jx\in C_\delta$, and called {\it the first return
time of $x$} (even if it is $0$). Given $n_k$ with $f^{n_k}x\in
C_\delta$, $p_k$ is the bound period and $n_{k+1}$ is the smallest
$j\geq n_k+p_k$ such that $f^jx\in C_\delta$. This decompose the
orbit of $x$ into bound segments corresponding to time intervals
$(n_k,n_k+p_k)$ and free segments corresponding to time intervals
$[n_k+p_k,n_{k+1}]$. The times $n_k$ are called {\it free return
times}.

\subsection{A few estimates}\label{few}
We quote from \cite{TW09} some technical estimates which will be
used in Section 3. Let $x\in S^1\setminus (C\cup S)$ make a free
return at $\nu>0.$ Let $0\leq n_1<n_2<\cdots<n_t<\nu$ denote all
the free returns before $\nu.$ Let $p_1,p_2,\cdots,p_t$ denote the
corresponding bound periods. For each $k\in[1,t]$, let
$$
\Theta_{k}(x)=\sum_{i=n_k}^{n_k+p_k-1}d_i^{-1}(x) \ \ \ \text{ and
} \ \ \ \Theta_{0}(x)=\sum_{i=0}^{\nu-1}d_i^{-1}(x)
-\sum_{k=1}^{t} \Theta_{k}(x).
$$
The quantity $\Theta_{k}$ is the contribution of the bound segment
from $n_k$ to $n_k + p_k-1$ to the total distortion and $\Theta_0$
is the contribution of all free segments to the total distortion.
It is understood that if $\nu$ is the first return time to
$C_\delta$, then the second summand in the definition of
$\Theta_0(x)$ is $0$.

The following two estimates were obtained in the proof of {\rm
[\cite{TW09} Lemma 1.8.]}, when $x$ is a critical value. It is not
hard to see that, the same estimates hold for a general $x$:
\begin{equation} \label{sublem2}
|(f^{n_k+p_k})'x|^{-1} \Theta_{k}(x) \leq
|d_C(f^{n_k}x)|^{-\frac{18\alpha}{\lambda}};
\end{equation}
\begin{equation}\label{exp0}
|(f^{\nu})'x|^{-1} \Theta_0(x) < \frac{1}{\delta^{\frac{1}{3}}}.
\end{equation}

\begin{definition}
{\rm Let $x\in S^1\setminus (C\cup S)$. We say $\nu$ is a {\it
deep return time} of $x$ if it is the first return time of $x$ to
$C_\delta$, or else, for every free return $n_k<\nu$,
$1\leq k\leq t$, 
\begin{equation}\label{inessential}
2\log d_C(f^{\nu}x)+\sum_{n_j\in(n_k,n_t]}2\log d_C(f^{n_j}x)\leq
\log d_C(f^{n_k}x).
\end{equation}

We say $\nu$ is a {\it shallow return time} of $x$ if it is not a
deep return time.}
\end{definition}

\begin{lemma}\label{exp}
Let $\nu>0$ be a deep return time of  $x\in S^1\setminus (C\cup
S)$. Then
$$ |(f^{\nu})'x| \cdot D_{\nu}(x) \geq \sqrt{d_C(f^{\nu}x)}.$$
\end{lemma}
Lemma \ref{dist} gives a bounded distortion of $f^\nu$ on the
interval $[x-D_\nu(x),x+D_\nu(x)]$. Hence, Lemma \ref{exp} gives a
lower estimate of the length of the interval
$f^\nu([x-D_\nu(x),x+D_\nu(x)])$ in terms of the distance of
$f^{\nu}(x)$ to the critical set. It follows that this interval
contains a critical point to which $f^\nu(x)$ is bound.

The following estimate, obtained in the proof of [\cite{TW09}
Proposition 2.1], bounds a contribution from shallow returns by
that of deep returns.
\begin{lemma}\label{f4-s3.1}
Let $0\leq n_1<n_2<\cdots<n_t$ denote all the free return times of
$x\in S^1\setminus (C\cup S)$ up to time $n_t$. Then
\begin{equation*}
\sum_{n_1\leq n_j\leq n_t\colon\text{shallow return}} \log
d_C(f^{n_j}x)\geq \sum_{n_1\leq n_j\leq n_t\colon\text{deep
return}}\log d_C(f^{n_j}x).
\end{equation*}
\end{lemma}

\section{Inducing to a large scale}
Let
\begin{equation}\label{la}
M_0:=\left[\frac{2}{\lambda}\log(1/\delta)\right]=\left[
\frac{2\alpha N_0}{\lambda}\right],\end{equation} where the square
bracket denotes the integer part. Let $|$ $\cdot$ $|$ denote the
one-dimensional Lebesgue measure. In this section we prove
\begin{prop}\label{escape} For an arbitrary interval $I$ with
$\frac{\delta}{10}\leq|I|\leq \delta$, there exists a countable
partition $\mathcal P$ of $I$ into intervals and a stopping time
function $S\colon\mathcal P\to\{n\in\mathbb N\colon n\geq M_0\}$
such that:
\smallskip

\noindent{\rm (a)} for each  $\omega\in\mathcal P$,
$|f^{S(\omega)}\omega|\geq\sqrt{\delta}$ and
$|(f^{S(\omega)})'|\omega|\geq1/\delta^{\frac{1}{3}}>1$;

\noindent{\rm (b)} the distortion of $f^{S(\omega)}|\omega$ is
uniformly bounded. More precisely, for all $x,y\in\omega$,
\begin{equation*}\label{diz}
\left|\frac{|(f^{S(\omega)})'x|}{|(f^{S(\omega)})'y|}-1\right|
\leq
\frac{1}{\sqrt{\delta}}|f^{S(\omega)}x-f^{S(\omega)}y|;\end{equation*}

\noindent{\rm (c)}$|\{S\geq n\}|\leq \delta^{\frac{11}{12}}
L^{-\frac{\lambda n}{24}}$ holds for every $n>0$. Here, $\{S\geq
n\}$ is the union of all $\omega\in\mathcal P$ such that
$S(\omega)\geq n$.
\end{prop}

 In Section \ref{srb} we define and describe the combinatorics of
the partition $\mathcal P$ and the stopping time $S$. (a) (b)
follow from these definitions. In Section \ref{expon} we prove
(c), assuming some key estimates on the measure of a set with a
given combinatorics. In Section \ref{pro} we prove this key
estimate.

\subsection{Combinatorial structure}\label{srb}
For each $n\geq0$, considering $n$-iterates we construct a mod $0$
partition $\widehat{\mathcal P}_n$ of $I$. This construction is
designed so that: each element of $\mathcal P$ is an element of
some $\widehat{\mathcal P}_n$; $\omega\in\mathcal
P\cap\widehat{\mathcal P}_n$, if and only if $S(\omega)=n$.
\smallskip

 Let $\widehat{\mathcal
P}_0=\{I\}$, the trivial partition of $I.$ Let $n\geq 1$ and
$\omega\in\widehat{\mathcal P}_{n-1}$. Then $\widehat{\mathcal
P}_n|\omega$ is defined as follows:
\smallskip

\noindent{\it Case I: $f^{n-1}\omega$ does not meet $C\cup S$.} We
cut $\omega$ from the left to the right, so that each subinterval
has the form $[x,x+D_n(x)]$. If the rightmost interval does not
have this form, then we take it together with the adjacent
interval.
\smallskip

\noindent{\it Case II: $f^{n-1}\omega$ meets $C\cup S$.} Consider
a subinterval of $\omega$ whose $f^{n-1}$-image does not meet
$C\cup S$ in its interior. Let $\omega'$ denote any maximal
interval with this property. We cut the right half of $\omega'$
from the left to the right, as in Case I. We cut the left half of
$\omega'$ from the right to the left, analogously to Case I.
\smallskip

Let us record some basic properties of the partitions.
\smallskip

\noindent{\it {\rm (P1)} Non-triviality}. For every $n\geq M_0$,
$\widehat{\mathcal P}_{n}\neq\{I\}$ holds. Indeed, if $I\cap
C_\delta\neq\emptyset$, then $D_1(x)< d_S(x)/\sqrt{L}\ll\delta/10$
holds for $x\in I\cap C_\delta$, while $|I|\geq\delta/10$ by the
assumption. Hence $I$ is subdivided in the construction of
$\widehat{\mathcal P}_1$, that is $I\notin\widehat{\mathcal P}_1$.
If $I\cap C_\delta=\emptyset$, then by Lemma \ref{outside}, either
(i) there exists $n\leq M_0$ such that $f^iI\cap
C_\delta=\emptyset$ for every $0\leq i<n$ and $f^nI\cap
C_\delta\neq\emptyset$, or else (ii) there exists $n\leq M_0$ such
that $f^n(\omega)=S^1$, by the nature of the singularities. If (i)
holds, then the same reasoning to the first case gives
$I\notin\widehat{\mathcal P}_{n+1}$. If (ii) holds, then clearly
$I\notin\widehat{\mathcal P}_{n+1}$.

\smallskip

\noindent{\it {\rm (P2)} Bounded distortion.} By (P1), if $n\geq
M_0$, then for each $\omega\in\widehat{\mathcal P}_n$, $D_n(x)\leq
|\omega|\leq10D_{n}(x)$ holds for some $x\in\omega$. From Lemma
\ref{dist} (see Remark \ref{south}), the distortion of
$f^n|\omega$ is bounded.
\smallskip

\noindent{\it {\rm (P3)} Uniform expansion.} Let $n\geq M_0$,
$\omega\in\widehat{\mathcal P}_n$ and suppose that
$|f^n\omega|\geq\sqrt{\delta}$. From $|\omega|\leq\delta$ and the
second estimate in Lemma \ref{dist},
\begin{equation}\label{expa}
|(f^n)'x|\geq\frac{1}{2}\frac{|f^n\omega|}{|\omega|}\geq
\frac{1}{2\sqrt{\delta}}>\frac{1}{\delta^{\frac{1}{3}}}
\quad\forall x\in\omega.\end{equation}
\smallskip

\begin{definition}\label{escd}
{\rm Given $\omega_{n}\in\widehat{\mathcal P}_{n}$ and
$k\in[0,n)$, let $\omega_k$ denote the unique element of
$\widehat{\mathcal P}_{k}$ which contains $\omega_{n}$. Let $n\geq
M_0$. We say $\omega_{n}\in\widehat{\mathcal P}_{n}$ reaches a
large scale at time $n$ if
$$n=\min\left\{
i\in[M_0,n]\colon |f^i\omega_{i}|\geq\sqrt{\delta}\right\}.$$}
\end{definition}

Let $\mathcal P_n$ denote the collection of all elements of
$\widehat{\mathcal P}_n$ which reach a large scale at time $n$.
Let $\mathcal P=\bigcup_n\mathcal P_n$. Define a stopping time
function $S\colon \mathcal P \to\mathbb N$ by $S(\omega)=n$ for
each $\omega\in\mathcal P_n.$ Let $\{S\geq n\}$ denote the union
of all $\omega\in\mathcal P$ such that $S(\omega)\geq n$. Let
$$\mathcal P_n'=\left\{ \omega_n\in\widehat{\mathcal P}_n\colon  |f^i\omega_{i}|<\sqrt{\delta}
\quad M_0\leq \forall i \leq n\right\},$$ and let
$|P_n|=\sum_{\omega\in\mathcal P_n'}|\omega|$. To show that
$\mathcal P$ is a mod $0$ partition of $I$, it suffices to show
$|P_n|\to0$ as $n\to\infty$. Then, $|\{S\geq n\}|=|P_n|$ holds.
(a) follows from (P3). (b) follows from the second estimate in
Lemma \ref{dist}.
\subsection{Exponential tails}\label{expon}
To prove (c), we have to show that $|P_n|$ decays exponentially.
\begin{lemma}\label{escape2}
If $n\geq M_0$, $\omega\in\widehat{\mathcal P}_n$ and
$f^i(\omega)\cap C_\delta=\emptyset$ for every $0\leq i<n$, then
$|f^{n}\omega|\geq\sqrt{\delta}$.
\end{lemma}
\begin{proof}
(P1) gives $|\omega|\geq D_n(x)$ for some $x\in \omega.$
(\ref{exp0}) gives
$$|(f^n)'x|^{-1}|D_n(x)|^{-1}=
\sqrt{L}\cdot\sum_{i=0}^{n-1}|(f^{n})'x|^{-1}d_i(x)^{-1} \leq
\frac{\sqrt{L}}{\delta^{\frac{1}{3}}}.$$ Taking reciprocals and
then using the bounded distortion of $f^n|\omega$, we obtain the
inequality.
\end{proof}
To each $\omega_n\in\mathcal P_{n}'$ we assign an {\it itinerary}
$$\bold i=(\nu_1,r_1,c_{1}),
(\nu_2,r_2,c_{2}),\cdots,(\nu_q,r_q,c_{q})$$ which has the
following interpretation. Let $x_*$ denote the mid point of
$\omega_n.$ Then $0\leq\nu_1<\cdots<\nu_q< n$ are all the deep
returns of the orbit of $x_*$ before $n$; for each $i\in[1,q]$,
$f^{\nu_i}x_*$ is bound to $c_i\in C$ and $r_i$ is the unique
integer such that $|c_i-f^{\nu_i}x_*|\in(L^{-r_i},L^{-r_i+1}].$
Let $P_n(\bold i)$ denote the union of all elements of $\mathcal
P_{n}'$ with an itinerary $\bold i$.
Lemma \ref{escape2} gives $|\{S\geq n\}| =\sum_{\bold i}
|P_n(\bold i)|$, where the sum ranges over all feasible
itineraries.

\begin{lemma}\label{pro-R1}
$|P_n({\bf i})| < L^{-\frac{1}{3} R}$, where $R = r_1 + r_2 \cdots
+ r_q$.
\end{lemma}
We finish the proof of (c) assuming the conclusion of this lemma.
First, we count the number of all itineraries with the same $R$ as
follows. First, two consecutive returns to $C_\delta$ are
separated at least by $\alpha N_0$, and thus the largest possible
number of returns in the first $n$ iterates is $n/\alpha N_0$.
Second, given $q\in[1,n/\alpha N_0]$, there are at most
$\left(\begin{smallmatrix}n\\q\end{smallmatrix}\right)$ number of
ways to choose the positions of $q$ number of free returns in
$[0,n]$. For each such way $(n_1,\cdots,n_q)$ there is at most
$\left(\begin{smallmatrix}R+q\\q\end{smallmatrix}\right)$ number
of ways to assign $r_1,\cdots,r_q$ with $r_1+\cdots+r_q=R$. Hence
\begin{equation}\label{sn}
|\{S\geq n\}|=\sum_{R}\sum_{\stackrel{P_n(\bold
i)}{r_1+\cdots+r_q=R}}|P_{n}(\bold i)|=
\sum_{R}\sum_{q=1}^{R/\alpha N_0}
\begin{pmatrix}n\\q\end{pmatrix}\begin{pmatrix}R+q\\q\end{pmatrix}
L^{-\frac{R}{3}}\leq\sum_RL^{-\frac{R}{4}}.\end{equation} The last
inequality follows from Stirling's formula for factorials. 

 To get a lower bound on $R$, take one element
$\omega\in\mathcal P_n'$ with an itinerary $\bold i$ and let
$0\leq n_1<\cdots<n_t< n$ denote all the free (both shallow and
deep) returns of the mid point $x_*$ of $\omega$ before $n$. Let
$p_k$ denote the bound period for $n_k$ and $s_k$ the unique
integer such that $d_C(f^{n_k}x_*)\in(L^{-s_k},L^{-s_k+1}]$ holds.
\begin{lemma}\label{space0}
For every $1\leq k<t$, $n_{k+1}-n_k\leq \frac{3s_k}{\lambda}$.
\end{lemma}
\begin{proof}
We assume $n_{k+1}>n_k+\frac{3s_k}{\lambda}$ and derive a
contradiction. By the upper estimate of the bound period in Lemma
\ref{reclem1}, $n_{k+1}>n_k+p_k+\frac{s_k}{\lambda}$ holds. By
Lemma \ref{outside},
$$|(f^{n_{k+1}-n_{k}-p_k})'f^{n_k+p_k}x_*|\geq L^{s_k}
\geq|d_C(f^{n_k}x_*)|^{-1}.$$ For every $1\leq j\leq k,$
\begin{align*}
|(f^{n_{k+1}-n_{j}-p_j})'f^{n_j+p_j}x_*|&=
|(f^{n_{k+1}-n_{k}-p_k})'f^{n_k+p_k}x_*|
|(f^{n_{k}+p_k-n_{j}-p_j})'f^{n_j+p_j}x_*|\\
&\geq
|d_C(f^{n_k}x_*)|^{-1}L^{\frac{\lambda}{3}(n_{k}+p_k-n_{j}-p_j)},
\end{align*}
and therefore
\begin{align*}\sum_{i=n_j}^{n_j+p_j}|(f^{n_{k+1}})'x_*|^{-1}
d_i^{-1}(x_*)&=|(f^{n_{k+1}-n_{j}-p_j})'f^{n_j+p_j}x_*|^{-1}\cdot
\sum_{i=n_j}^{n_j+p_j}|(f^{n_j+p_{j}})'x_*|^{-1}d_i^{-1}(x_*)\\
& \leq L^{-\frac{\lambda}{3}(n_{k}+p_k-n_{j}-p_j)}
|d_C(f^{n_k}x_*)|^{1-\frac{18\alpha }{\lambda}}\leq
L^{-\frac{\lambda}{3}(n_{k}+p_k-n_{j}-p_j)}.\end{align*} For the
second factor in the right-hand-side of the equality we have used
(\ref{sublem2}). Summing this over all $1\leq j\leq k$ and adding
the contribution from all the free iterates outside of $C_\delta$
which was estimated in (\ref{exp0}), \begin{align*}
\sum_{i=0}^{n_{k+1}-1} |(f^{n_{k+1}
})'x_*|^{-1}d_i^{-1}(x_*)&=\sum_{j=1}^{k}\sum_{i=n_j}^{n_j+p_j}
+\sum_{i\in\cup_{j=1}^k(n_j+p_j,n_{j+1})}\\
&\leq\frac{1}{\delta^{\frac{1}{3}}}+\sum_{j=1}^k
L^{-\frac{\lambda}{3}(n_{k}+p_k-n_{j}-p_j)}\leq\frac{1}{\delta^{\frac{1}{3}}}+
\sum_{i=0}^\infty L^{-\frac{\lambda
i}{3}}<\frac{2}{\delta^{\frac{1}{3}}}.
\end{align*}
Taking reciprocal and then using the bounded distortion of
$f^{n_{k+1}}|\omega_{n_{k+1}}$, we have
$|f^{n_{k+1}}\omega_{n_{k+1}}|\geq\sqrt{\delta}$. This yields a
contradiction to the assumption $S(\omega)\geq n$.
\end{proof}

Summing the inequality in Lemma \ref{space0} over all $1\leq k<t$
gives \begin{equation}\label{lu} n_t\leq
n_1+\frac{3}{\lambda}\sum_{k=1}^{t-1}s_k.\end{equation} From this
point on we assume $n\geq 2M_0.$ Then $n_1\leq n/2$ holds, for
otherwise $n_1>n/2\geq M_0$ and
$|f^{n_1}\omega_{n_1}|\geq\sqrt{\delta}$ would follow from Lemma
\ref{escape2}, a contradiction to $S(\omega)\geq n$. We have
$$n_t+p_t\leq \frac{n}{2}+p_t+\frac{3}{\lambda}\sum_{k=1}^{t-1}s_k
\leq\frac{n}{2}+\frac{3}{\lambda}\sum_{k=1}^{t}s_k\leq\frac{n}{2}
+\frac{6}{\lambda}R.$$ We have used $p_t\leq\frac{2}{\lambda}s_t$
for the second inequality which follows from Lemma \ref{reclem1},
and Lemma \ref{f4-s3.1} for the last. When $n$ is bound, then
$n<n_t+p_t$ holds, and thus the above inequality yields
$R\geq\frac{\lambda n}{12}$. When $n$ is free, repeating the
argument in the proof of Lemma \ref{space0} we get
$n-n_t\leq\frac{3s_t}{\lambda}$. Combining this with (\ref{lu})
yields the same lower bound of $R$. Consequently we obtain
$|\{S\geq n\}|\leq L^{-\frac{\lambda n}{24}}$ for every $n\geq
2M_0$. As $|I|\leq\delta$, $|\{S\geq n\}|\leq \delta
L^{\frac{2\lambda M_0}{24}}L^{-\frac{\lambda n}{24}}$ holds for
every $n>0$. The choice of $M_0$ in (\ref{la}) gives $
L^{\frac{\lambda M_0}{12}}\leq \delta^{-\frac{1}{12}}$, and the
desired inequality holds.

\subsection{Proof of Lemma \ref{pro-R1}}\label{pro}
We first treat the case $\nu_1>0$. For all $x\in P_n(\bold i)$ and
each $k\in[1,q]$ we define an interval $I_k(x)$ in such a way that
$f^{\nu_k}$ sends $I_k(x)$ to an interval injectively with bounded
distortion.
Let $\hat I_k(x)$ denote the interval of length $D_{\nu_k}(x)$
centered at $x$. By Lemma \ref{dist}, the distortion of
$f^{\nu_k}|\hat I_k(x)$ is uniformly bounded, while $f^{\nu_k}$
may not be injective on $\hat I_k(x).$ If $|f^{\nu_k}(\hat
I_k(x))|\leq1/4$, define $I_k(x)= \hat I_k(x)$. Otherwise, define
$I_k(x)$ to be the interval of length $ |\hat
I_k(x)|/(10|f^{\nu_k}(\hat I_k(x))|)$ centered at
$x$. By construction, $f^{\nu_k}$ is injective on $I_k(x).$ 
\smallskip

\noindent{\it Notation.} For a compact interval $I$ centered at
$x$ and $r>0$, let $ r\cdot I$ denote the interval of length
$r|I|$ centered at $x$.
\smallskip

For each $k\in[1,q]$, we choose a subset (possibly infinite)
$\{x_{k,i}\}_{i}$ of $P_n(\bold i)$ with the following properties:
\smallskip

\noindent(i) the intervals $\{I_{k}(x_{k,i})\}_{i}$ are pairwise
disjoint and $P_n(\bold i)\subset \bigcup _{i}L^{-r_k/3}\cdot
{I}_{k}(x_{k,i});$

\noindent(ii) for each $k\in[2,q]$ and $x_{k,i}$ there exists
$x_{k-1,j}$ such that ${I}_{k}(x_{k,i})\subset
2L^{-r_{k-1}/3}\cdot{I}_{{k-1}}(x_{k-1,j})$.
\smallskip

Let $M_k=\sum_{i}|{I}_{k}(x_{k,i})|.$ It follows that $|P_n(\bold
i)|\leq L^{-r_q/3}M_q$ and $M_k\leq L^{-r_{k-1}/3}M_{k-1},$ and
therefore $|P_n(\bold i)|\leq L^{-\frac{1}{3}\sum_{k=1}^q r_k},$
and the desired estimate holds.
\smallskip

For the definition of the subsets we need two combinatorial
lemmas.
 The following elementary fact from
Lemma \ref{exp} is used in the proofs of these two lemmas:
$(f^{\nu_k}|I_k(x))^{-1}(c_k)$ consists of a single point and
$(f^{\nu_k}|I_k(x))^{-1}(c_k)\subset L^{-r_k/3}\cdot I_{k}(x)$.

\begin{lemma}\label{lem1}
If $x$, $y\in P_n(\bold i)$ and $y\notin I_{k}(x)$, then
$I_{k}(x)\cap I_{k}(y)=\emptyset$ .
\end{lemma}

\begin{proof}
Suppose $I_k(x)\cap I_k(y)\neq\emptyset$. Lemma \ref{dist} gives
$|I_{k}(x)|\approx |I_{k}(y)|$. This and $y\in I_k(x)$ imply
$(f^{\nu_k}|I_k(x))^{-1}(c_k)\neq(f^{\nu_k}|I_k(y))^{-1}(c_k).$ On
the other hand, by the definition of the intervals $I_{k}(\cdot)$,
$f^{\nu_k}$ is injective on $I_{k}(x)\cup I_{k}(y)$. A
contradiction arises.
\end{proof}

\begin{lemma}\label{lem2} If $x$, $y\in P_n
(\bold i)$
 and
$y\in L^{-r_k/3}\cdot I_{k}(x)$, then $I_{{k+1}}(y)\subset
2L^{-r_k/3}\cdot I_{k}(x)$.
\end{lemma}

\begin{proof}
We have $(f^{\nu_k}|I_k(x))^{-1}(c_k)\notin I_{k+1}(y)$, for
otherwise the distortion of $f^{\nu_{k+1}}|I_{k+1}(y)$ is
unbounded. This and the assumption together imply that one of the
connected components of $I_{k+1}(y)-\{y\}$ is contained in
$L^{-r_k/3}\cdot I_{k}(x)$. This implies the inclusion.
\end{proof}

We are in position to choose subsets $\{x_{k,i}\}_{i}$ satisfying
(i) (ii). Lemma \ref{lem1} with $k=1$ allows us to pick a subset
$\{x_{1,i}\}$ such that the corresponding intervals $
\{I_{1}(x_{1,i})\}$ are pairwise disjoint, and altogether cover
$P_n(\bold i)$. Indeed, pick an arbitrary $x_{1,1}$. If $
I_{1}(x_{1,1})$ covers $P_n(\bold i)$, then the claim holds.
Otherwise, pick $x_{1,2}\in P_n(\bold i)-{I}_{1}(x_{1,1})$. By
Lemma \ref{lem1}, $I_{1}(x_{11}),{I}_{1}(x_{12})$ are disjoint.
Repeat this. By Lemma \ref{lem1}, we end up with pairwise disjoint
intervals. To check the inclusion in (i), let $x\in
{I}_{1}(x_{1i})-L^{-r_1/3}\cdot{I}_{1}(x_{1i})$. By Lemma
\ref{exp}, $|f^{\nu_1}x-c_{1}|\gg L^{-r_k}$ holds. Hence $x\notin
P_n(\bold i)$.

Given $\{x_{k-1,j}\}_j$, we choose $\{x_{k,i}\}_i$ as follows. For
each $x_{k-1,j}$, similarly to the previous paragraph it is
possible to choose parameters $\{x_{m}\}_m$ in $P_n(\bold i )\cap
L^{-r_{k-1}/3}\cdot{I}_{{k-1}}(x_{k-1,j})$ such that the
corresponding intervals $\{I_k(x_m)\}_m$ are pairwise disjoint and
altogether cover $P_n(\bold i )\cap
L^{-r_{k-1}/3}\cdot{I}_{{k-1}}(x_{k-1,j})$. In addition, Lemma
\ref{lem2} gives $\bigcup_m{I}_{k}(x_{m})\subset 2L^{-r_{k-1}/3}
 \cdot{I}_{{k-1}}(x_{k-1,j}).$ Let $\{x_{k,i}\}_{i}=
\bigcup_{j}\{x_{m}\}$.

It is left to treat the case $\nu_1=0$. In this particular case,
by definition of $\bold i$, $P_n(\bold i)$ is contained in
$(-L^{-r_1+1},L^{-r_1+1})$. Hence, the desired estimate holds if
$q=1$. If $q>1$, then in the same way as above, it is possible to
show $|P_n(\bold i )|\leq L^{-\frac{1}{3}(R-r_1)}2L^{-r_1+1}$,
which is $\leq L^{-\frac{R}{3}}$. This finishes the proof of Lemma
\ref{pro-R1}. \qed

\section{Induced Markov map on $S^1$}
In this section we construct an induced Markov map on $S^1$ and
complete the proofs of the theorems.
\begin{prop}\label{Mar}
There exist a partition $\mathcal Q$ of a full measure set of
$S^1$ into a countable number of open intervals and a return time
function $R\colon\mathcal Q\to\{n\in\mathbb N\colon n>M_0\}$ with
the following properties. For each $\omega\in\mathcal Q$, $F:=f^R$
sends $\omega$ injectively, so that $\overline{F(\omega)}=S^1$.
There exists $K>0$ such that for all $\omega\in\mathcal Q$ and all
$x$, $y\in\omega$,
\begin{equation}\label{equation}
\left|\frac{|F'(x)|}{|F'(y)|}-1\right|\leq
K|F(x)-F(y)|.\end{equation} In addition, $|\{R=n\}|\leq \delta
L^{-\frac{\lambda n}{26}}$ holds for every $n>M_0$. Here,
$\{R=n\}$ denotes the union of $\omega\in\mathcal Q$ such that
$R(\omega)=n$.
\end{prop}

Our inducing time consists of four explicit parts: the first part
is used to recover from the small derivatives near the critical
set (Proposition \ref{reclem1}); in the second, intervals reach a
``large scale scale" (Proposition \ref{escape}) and in the third
they reach a neighborhood of the critical set. The last part is
used to completely ``wrap" the circle.

In Section \ref{chu1} we prove a key lemma used in the third and
fourth parts of the inducing time. In Section \ref{full} we
construct the induced map $F$ with the desired properties. In
Section \ref{A} we prove Theorem A. In Section \ref{B} we prove
Theorem B.

\subsection{Inducing to the entire $S^1$}\label{chu1}
We show that intervals with scale $\sqrt{\delta}$ soon grow to the
entire $S^1$. There are two scenarios for this growth. One is to
take advantage of the nature of the singularities. The other is to
follow the initial iterates of the critical orbits, which are kept
out of $C_\sigma, S_\sigma$ by the standing hypothesis (a) in
Sect.\ref{gs}.

We first show that intervals with scale $\sqrt{\delta}$ soon reach
critical or singular neighborhoods.
\begin{lemma}\label{mix'} For any interval $\omega$ of length $\geq\sqrt{\delta}/3$,
there exist a subinterval $\omega'$ and an integer $M\leq M_0$
such that $d_C(f^i\omega')\geq\delta,$ $d_S(f^i\omega')\geq\delta$
for every $0\leq i<M$ and $f^M(\omega')$ coincides with one
component of $C_\delta\cup S_\delta$.
\end{lemma}
\begin{proof}
 We iterate $\omega$, deleting
all parts that fall into $C_\delta\cup S_\delta$. Suppose that
this is continued up to step $n$, and that for every $i\leq n$,
none of these deleted segments is $<2\delta$ in length. By the
assumption, the number of deleted segments at step $i\leq n$ is
$\leq 2$. By Lemma \ref{outside}, all deleted parts in $\omega$
are $\leq4\delta\sum_{i=0}^nL^{-2\lambda i}$ in length. Hence, the
undeleted segment in $f^n\omega$ is $\geq
\left(\sqrt{\delta}/3-4\delta\sum_{i=0}^nL^{-\lambda
i}\right)\delta L^{2\lambda n}$ in length. It follows that before
step $M_0$ there must come a point when our claim is fulfilled.
\end{proof}

For convenience, let us introduce the following language.
\begin{definition}
{\rm Let $\varepsilon>0$ and $M>0$ an integer. A pair of open
intervals $(\omega,\tilde\omega)$ with $\tilde\omega\subset\omega$
is a {\it good $(\varepsilon,M)$-pair} if: (i)
$|\tilde\omega|\geq\varepsilon|\omega|$; (ii) $\omega\setminus
\tilde\omega$ has two components and their lengths are
$\geq\sqrt{\delta}/3$; (iii) $f^{M}$ is injective $\tilde\omega$
and $\overline{f^M(\tilde\omega)}=S^1$; (iv)
$d_C(f^i\tilde\omega)\geq\varepsilon,$
$d_S(f^i\tilde\omega)\geq\varepsilon$ for every $0\leq i<M$.}
\end{definition}

\begin{lemma}\label{mix} There
exists $0<\varepsilon_0<1$ such that for any interval $\omega$ of
length $\geq\sqrt{\delta}$, there exist a subinterval
$\tilde\omega$ in its middle third and an integer $k\leq 2M_0$
such that $(\omega,\tilde\omega)$ is a good
$(\varepsilon_0,k)$-pair.
\end{lemma}

\begin{proof}
Take a subinterval $\omega'$ in the middle third of $\omega$ and
an integer $M$ for which the conclusion of Lemma \ref{mix'} holds.
We deal with two cases separately.
\smallskip

\noindent{\it Case I: $f^{M}\omega'\subset S_\delta$.} By the
nature of the singularity, there exists a subinterval
$\omega''\subset f^{M} \omega'$ such that
$d_S(\omega'')\geq\delta/10$, $f|\omega''$ is injective and
$\overline{f(\omega'')}=S^1$. Let $\tilde\omega=f^{-M}(\omega'')$
and $k=M+1$.
\smallskip

\noindent{\it Case II: $f^{M}\omega'\subset C_\delta$.} Let
$N_1=\left[10\alpha N_0\right].$ Let $c$ denote the critical point
in $f^{M}\omega'$. By the definition of $\delta$, $f^{M}\omega'$
contains $I_{N_1}(c)$.
\begin{sublemma}\label{rap}
$f^{N_1+1}(I_{N_1}(c))=S^1$, and $d_C(f^iI_{N_1}(c))\geq\sigma/2$,
$d_S(f^iI_{N_1}(c))\geq\sigma/2$ for every $1\leq i\leq N_1$.
\end{sublemma}
We finish the proof of Lemma \ref{mix} assuming the conclusion of
this sublemma. Take a subinterval $J\subset I_{N_1}(c)$ on which
$f^{N_1+1}$ is injective and $\overline{f^{N_1+1}(J)}=S^1$ holds.
Let $\tilde \omega=f^{-M}(J)$, and $k=M+N_1+1$. Sublemma \ref{rap}
gives $d_C(f^i\tilde\omega)\geq\sigma/2,$
$d_S(f^i\tilde\omega)\geq\sigma/2$ for every $M<i<k$. As
$f^M(\tilde\omega)\subset I_{N_1}(c)$, $d_C(f^M\tilde\omega)\geq
\sqrt{D_{N_1+1}(c)/(K_0L)}$ holds. (\ref{la}) gives $k\leq 2M_0$.
\smallskip

In either of the two cases, the $M$-iterates of $\omega'$ are kept
out of $C_\delta\cup S_\delta$. Hence, the distortion of
$f^{M}|\omega'$ is uniformly bounded and there exists a uniform
constant $0<\varepsilon_0'<1$ such that
$|\tilde\omega|\geq\varepsilon_0'|\omega'|$. Set
$$\varepsilon_0=\min\left(\delta/10,\sigma/2,\inf_{c\in C
}\sqrt{D_{N_1+1}(c)/(K_0L)},\varepsilon_0'\right).$$ 
Then, $(\omega,\tilde\omega)$ is a good $(\varepsilon_0,k)$-pair.
\smallskip

It is left to prove Sublemma \ref{rap}. Let $f^{i+1}(c)=v_i$. The
next standing hypothesis (a) in Sect.\ref{gs} is used:
$d_C(v_i)\geq\sigma$, $d_S(v_i)\geq\sigma$ for every $0\leq
i<N_0$.

Let $l=\min\{|x-y|\colon x\in C,y\in S\}>0$. It is easy to see
that, $d_C(v_i)\geq\sigma, d_S(v_i)\geq\sigma$ give
 $$d_C(v_i)d_S(v_i)\geq \sigma(l-\sigma)\geq l\sigma/2,$$
 where the last inequality holds for sufficiently large $L$.
In view of (\ref{Theta}) and
$|(f^{N_1})'(v_0)|\geq(L\sigma/K_0)^{N_1-i} |(f^i)'(v_0)|$,
$$\frac{1}{\sqrt{L}}\frac{d_{N_1}(v_0)}{D_{N_1}(v_0)}
= \sum_{i=0}^{N_1-1}\frac{|(f^i)'(v_0)|}{|(f^{N_1})'(v_0)|}
\frac{d_C(v_{N_1})d_S(v_{N_1})}{d_C(v_i)d_S(v_i)}
\leq\sum_{i=0}^{N_1-1}\frac{2}{(L\sigma/K_0)^{N_1-i}\sigma l}
\leq\frac{1}{\sqrt{L}}.$$ This yields
$D_{N_1}(v_0)/D_{N_1+1}(v_0)=1+\sqrt{L}D_{N_1}(v_0)d_{N_1}^{-1}(v_0)
\geq2.$ We also have
$$|(f^{N_1})'(v_0)D_{N_1}(v_0)|^{-1}
=\sqrt{L}\sum_{i=0}^{N_1-1}\frac{|(f^i)'(v_0)|}{|(f^{N_1})'(v_0)|}
\frac{1}{d_C(v_i)d_S(v_i)}\leq\frac{2}{\sqrt{L}\sigma^2l
(1-1/(L\sigma/K_0))} \leq L^{-\frac{1}{7}},$$ where the last
inequality follows from $\sigma=L^{-\frac{1}{6}}$. Hence
\begin{align*}|f^{N_1+1}(I_{N_1}(c))|&\geq
\frac{1}{2}|(f^{N_1})'v_0||f(I_{N_1}(c))|
\geq \frac{1}{2}|(f^{N_1})'v_0|(D_{N_1}(v_0)-D_{N_1+1}(v_0))\\
&\geq \frac{1}{4}|(f^{N_1})'D_{N_1}(v_0)|\geq
L^{\frac{1}{8}}\gg1.\end{align*} Hence, the first claim holds. The
second follows from the standing hypothesis and
$|f^i(I_{N_1}(c))|\leq |(f^{i})'v_0|D_{N_1}(v_0)\leq
\frac{1}{\sqrt{L}}|(f^{i})'v_0|d_i(v_0)\leq\frac{1}{\sqrt{L}}.$
This completes the proof of Sublemma \ref{rap}.
\end{proof}

\subsection{Full return map}\label{full}
We now define a partition $\mathcal Q$ of $S^1$ and a return time
function $R\colon \mathcal Q\to \mathbb N$. First of all, cut
$S^1$ into pairwise disjoint intervals of lengths from $\delta/10$
to $\delta$. For each interval, consider its partition $\mathcal
P$ and the associated return time function $S\colon \mathcal
P\to\mathbb N$, given by Proposition \ref{escape}. By Lemma
\ref{mix}, for each $\omega_1\in\mathcal P$ there exists a
subinterval $\tilde\omega_1$ and an integer $M_1$ such that
$(f^{S(\omega_1)}\omega_1,f^{S(\omega_1)}\tilde\omega_1)$ is a
good $(\varepsilon_0,M_1)$-pair. Let $\tilde\omega_1\in\mathcal Q$
and $R(\tilde\omega_1):=S(\omega_1)+M_1$. Each component of
$\omega_1\setminus\tilde\omega_1$ is said to {\it have $1$ large
scale times}.

Subdivide each component of $f^{S(\omega_1)}(\omega_1)\setminus
f^{S(\omega_1)}(\tilde\omega_1)$, which is of length
$\geq\sqrt{\delta}/3$ by the definition of good pairs, into
intervals of lengths from $\delta/10$ to $\delta$. To each
interval, consider again its partition $\mathcal P$ and the
stopping time function $S$ given by Proposition \ref{escape}. For
each element $\omega_2$ of the partition, there exist an integer
$M_2$ and a subinterval $\tilde \omega_2$ such that
$(f^{S(\omega_2)}\omega_2,f^{S(\omega_2)}\tilde \omega_2)$ is a
good $(\varepsilon_0,M_2)$-pair. Let $f^{-S(\omega_1)}(\tilde
\omega_2)\in\mathcal Q$ and $R(\tilde \omega_2)
:=S(\omega_1)+S(\omega_2)+M_2$. Each component of
$f^{-S(\omega_1)}\omega_2\setminus f^{-S(\omega_1)}\tilde\omega_2$
is said to {\it have $2$ large scale times}, and so on.
\smallskip

\subsubsection{Bounded distortion}
We verify (\ref{equation}). By construction, for each
$\omega\in\mathcal Q$ there exists an associated sequence of large
scale times
$$0=S_0<S_1<S_2<\cdots<S_{q(\omega)}<R(\omega)$$
with $R(\omega)=S_{q(\omega)}+t(\omega)$ and $t(\omega)\leq 2M_0$.
For each $0\leq i<q$,
$|(f^{S_{i+1}-S_i})'|\geq1/\delta^{\frac{1}{3}}$ holds on
$f^{S_i}\omega$. The second estimate in Lemma \ref{dist} gives
\begin{equation}\label{total}
{\rm ln}\frac{|(f^{S_{i+1}-S_i})'(f^{S_i}x)|}
{|(f^{S_{i+1}-S_i})'(f^{S_i}y)|}\leq
\frac{1}{\sqrt\delta}|f^{S_{i+1}}(x)-
f^{S_{i+1}}(y)|\leq\frac{\delta^{\frac{q-i-1}{3}}}{\sqrt\delta}
|f^{S_{q}}(x)- f^{S_{q}}(y)|.\end{equation} The additional at most
$2M_0$ iterates after the last large scale time $S_{q}$ does not
significantly affect the distortion. Consequently,
(\ref{equation}) holds.

\subsubsection{Exponential tails}
For each $1\leq i< n$, let $\mathcal Q_n^{(i)}$ denote the
collection of all $\omega\in\mathcal Q$ which have exactly $i$
large scale times before $n$ and $R(\omega)=n$. Let $|\mathcal
Q_n^{(i)}|=\sum_{\omega\in\mathcal Q_n^{(i)}}|\omega|.$ By
construction, two consecutive large scale times are separated at
least by $M_0$, and $|\{R=n\}|=\sum_{1\leq i\leq n/M_0}|\mathcal
Q_n^{(i)}|$ holds.

We estimate the measure of $\mathcal Q_n^{(i)}$. For $2\leq i\leq
n/M_0$ and an $i$ string $(k_1,\cdots,k_i)$ of positive integers
with $k_1+\cdots+k_i< n$, let $$\mathcal
Q_n(k_1,\cdots,k_i)=\left\{\omega\in \mathcal Q^{(i)}_{n}\colon
S_j-S_{j-1}=k_j\ \ 1\leq \forall j\leq i\right\}.$$ Let $|\mathcal
Q_n(k_1,\cdots,k_i)| =\sum_{\omega\in\mathcal
Q_n(k_1,\cdots,k_i)}|\omega|$. For each $\omega\in\mathcal
Q_n(k_1,\cdots,k_{i-1})$, let
$$\mathcal Q_n(\omega,k_i)=
\{\omega'\in\mathcal Q_n(k_1,\cdots,k_i)\colon
\omega'\subset\omega\}.$$ By definition,
\begin{align*}
|\mathcal Q_n(k_1,\cdots,k_i)|&=\sum_{\omega\in \mathcal
Q_n(k_1,\cdots,k_{i-1})} |\omega|\sum_{\omega'\in\mathcal
Q_n(\omega,k_i) } \frac{|\omega'|} {|\omega|}.\end{align*} To
estimate the fraction, let $\omega\in \mathcal
Q_n(k_1,\cdots,k_{i-1})$. Proposition \ref{escape} gives
$$\sum_{\omega'\in\mathcal Q_n(\omega,k_i) }
|f^{k_1+\cdots+k_{i-1}}(\omega')|\leq
\delta^{\frac{11}{12}}L^{-\frac{\lambda
k_i}{24}}|f^{k_1+\cdots+k_{i-1}}(\omega)|.$$ By construction,
$|f^{k_1+\cdots+k_{i-1}}(\omega)|\leq\delta$ holds. By
(\ref{total}), the distortion of $f^{k_1+\cdots+k_{j-1}}|\omega$
is uniformly bounded and
$$\sum_{\omega'\in\mathcal Q_n(\omega,k_i)
}\frac{|\omega'|}{|\omega|} \leq
2\delta^{\frac{11}{12}}L^{-\frac{\lambda k_i}{24}}.$$ Hence we
obtain $|\mathcal Q_n(k_1,\cdots,k_i)|\leq L^{-\frac{k_i}{24}}
|\mathcal Q_n(k_1,\cdots,k_{i-1})|.$ Using this inductively and
then $|\mathcal Q_n(k_1)|\leq
\delta^{\frac{11}{12}}L^{-\frac{\lambda k_1}{24}}$ which follows
from Proposition \ref{escape},
\begin{equation*}
|\mathcal Q_n(k_1,\cdots,k_i)| \leq \delta^{\frac{11 }{12}}
L^{-\frac{1}{24}(k_1+\cdots+k_i)}.\end{equation*}

For any given $m\in[n-2M_0,n)$ and $1\leq i\leq n/M_0$, the number
of all feasible $(k_1,\cdots,k_i)$ with $k_1+\cdots+k_i=m$ equals
the number of ways of dividing $m$ objects into $i$ groups, which
is $\left(\begin{smallmatrix}m+i\\i\end{smallmatrix}\right)$, and
by Stiring's formula for factorials, this number is $\leq e^{\beta
m}$, where $\beta\to0$ as $L\to\infty$. Hence we obtain
\begin{align*}
|\mathcal Q_n^{(i)}|&=
\sum_{m=n-2M_0}^{n-1}\sum_{k_1+\cdots+k_i=m}|\mathcal
Q_n(k_1,\cdots,k_i)|\\&\leq \delta^{\frac{11
}{12}}\sum_{m=n-2M_0}^{n-1}L^{-\frac{m}{24}}\sharp\left\{
(k_1,\cdots,k_i)\colon \sum_{j=1}^i k_j=m\right\}\leq
\delta^{\frac{11 }{12}}L^{-\frac{\lambda}{25}(n-2M_0)}.
\end{align*}
The same inequality remains to hold for $i=1$. Summing these over
all $1\leq i\leq n/M_0$,
$$|\{R=n\}|=\sum_{1\leq i\leq n/M_0}|\mathcal Q_n^{(i)}|\leq
\delta^{\frac{11 }{12}}L^{\frac{2\lambda M_0}{25}}
L^{-\frac{\lambda n}{26}}\leq\delta L^{-\frac{\lambda n}{26}}.$$
The last inequality follows from $L^{\frac{\lambda M_0}{12}}\leq
\delta^{\frac{1}{6}}$. This finishes the proof of (c).

\subsection{Proof of Theorem A}\label{A}
We have constructed an induced Markov map with exponential tails
and Lipschitz bounded distortion property. Then all the statements
of Theorem A follow from the abstract scheme in \cite{Y2}. See
[\cite{BLS} pp.644] for a concise explanation.

In fact, to apply \cite{Y2}, the right hand side (\ref{equation})
has to be bounded by a uniform constant multiplied by
$\beta^{s(x,y)}$, where $0<\beta<1$ and $s(x,y)$ is a {\it
separation time} \cite{Y2}. This is a direct consequence of
(\ref{equation}) and the uniform expansion of $F$ on each
$\omega\in\mathcal Q$.

\subsection{Proof of Theorem B}\label{B}
Let $\phi\colon S^1\to\mathbb R$ be a H\"older continuous function
which is coboundary. Let $\psi\in L^2(\mu)$ satisfy
$\phi=\psi\circ f-\psi$. We show that $\psi$ has a version which
is (H\"older) continuous\footnote{The same conclusion follows from
Liv${\rm {\check s}}$ic regularity results \cite{BLS,G}.} on the
entire $S^1$.

For each $n\geq 1$, let $\mathcal H_n$ denote the collection of
inverse branches of $F^n$. Let $\mathcal F_n$ denote the
$\sigma$-algebra generated by the intervals $h(S^1)$ for
$h\in\mathcal H_n$. It is an increasing sequence of
$\sigma$-algebras. For almost every $x\in S^1$, there exists an
well-defined sequence $\bar h=(h_1,h_2,\cdots)\in\mathcal
H_1^{\mathbb N}$ such that the element $F_n(x)$ of $\mathcal F_n$
containing $x$ is given by $F_n(x)= h_1\circ\cdots \circ
h_n(S^1)$. Equivalently, $h_n$ is the unique element of $\mathcal
H_1$ such that $F^{n-1}(x)\in h_n(S^1)$.

The Martingale convergence theorem shows that, for almost every
$x\in S^1$ and for all $\epsilon>0$,
\begin{equation}\label{martin} \frac{{\rm Leb}\{x'\in F_n(x)\colon
|\psi(x')-\psi(x)|>\epsilon\}}{{\rm Leb}(F_n(x))} \to 0\quad a.e.\
x\in S^1.\end{equation} Take a point $x_0$ such that this
convergence holds. Let $\bar h=(h_1,h_2,\cdots)$ denote the
corresponding sequence of $\mathcal H$ and write $\bar
h_n=h_1\circ\cdots\circ h_n$, so that $F_n(x_0)=\bar h_n(S^1)$.
(\ref{martin}) and the bounded distortion of $\bar h_n$ give, for
all $\epsilon>0$,
\begin{equation*}\label{mart}{\rm Leb}\{x\in S^1\colon |\psi(\bar
h_nx)-\psi(x_0)|>\epsilon\}\to0.\end{equation*} Choose a
subsequence $(n_k)$ such that for all $\epsilon>0$,
$$\sum_{k=1}^\infty{\rm Leb}\{x\in S^1\colon |\psi(\bar h_{n_k}x)
-\psi(x_0)|>\epsilon \}<\infty.$$ By the Borel-Cantelli lemma,
$\psi(\bar h_{n_k}x)\to\psi(x_0)$ holds for almost every $x$.

Let $S_k(x)=\sum_{i=1}^{n_k} \phi(\bar h_{i}x)$. As
$\psi(x)=\psi(\bar h_{n_k}x)+S_{n_k}(x),$ $S_{k}(x)$ converges for
almost every $x$. For all $x$ such that this convergence holds,
let $S(x)=\lim_{k\to\infty}S_k(x)$. The uniform contraction over
all inverse branches and the H\"older continuity of $\phi$ give
$|S_k(x)-S_k(y)|\leq K|x-y|^{\eta}$, where $\eta$ is the H\"older
exponent of $\phi$. Passing to the limit we obtain
$|S(x)-S(y)|\leq K|x-y|^{\eta}$, that is, $S$ is continuous. As
$\psi(x)=\psi(x_0)+S(x)$, it follows that $\psi$ has a version
which is continuous on the entire $S^1$.

Assume that $\psi$ is not a constant function. Fix $z, z'$ such
that $\psi(z)\neq\psi(z')$. Fix a singular point $y\in S$. We
evaluate the cohomologous equation along a sequence $(x_n)$ with
$x_n\to y$. By continuity, $\phi(x_n)+\psi(x_n)\to\phi(y)+\psi(y)$
holds. To obtain a contradiction, it suffices to choose two
sequences $x_n\to y$, $x_n'\to y$ so that $\psi(fx_n),
\psi(fx_n')$ converge to different limits. By the nature of the
singularities, for any sufficiently large $n>0$ there exists an
interval $I_n$ in $[y,y+1/n]$ such that $f(I_n)=S^1$. Pick two
points $x_n\in f^{-1}(z)\cap I_n$, $x_n'\in f^{-1}(z')\cap I_n$.
Clearly, $x_n\to y,x_n'\to y$ and $\psi(fx_n)\to
\psi(z),\psi(fx_n')\to\psi(z')$ hold. This completes the proof of
Theorem B.
 \qed
\medskip

For later use in the next section, we prove
\begin{cor}\label{integrable}
$\log|f'|$ is $\mu$-integrable and $\bar\lambda:=\int\log
|f'|d\mu>0$.
\end{cor}
\begin{proof}
Let $\nu$ denote the acim of $F$. By the classical theorem, the
density of $\nu$ is uniformly bounded from above and below. From
the uniform expansion and the bounded distortion of $F$, there
exist $K>0$, $K'>0$ such that $K\leq \log\Vert F'|\omega \Vert\leq
K' |\log|\omega||$ holds for every $\omega\in\mathcal Q$. Choose
$0<\gamma<1$ such that $|\log|\omega||\leq |\omega|^{-\gamma}$
holds for every $\omega\in\mathcal Q$. Then
\begin{align*}
K\leq \int \log|F'|d\nu\leq K\sum_{\omega\in\mathcal Q} \int|\log
|\omega||(d\nu|\omega) \leq K\sum_{\omega\in\mathcal Q}
|\omega|^{1-\gamma}=K\sum_{n>0}\sum_{\omega\colon
R(\omega)=n}|\omega|^{1-\gamma},
\end{align*}
which is finite by (c) in Proposition \ref{Mar}. As
\begin{equation}\label{mu}
\mu=\frac{1}{\int Rd\nu}\sum_{\omega\in\mathcal
Q}\sum_{i=0}^{R(\omega)-1} (f_*^i)\nu|\omega,\end{equation} we
have
\begin{align*}
0<\int \log|F'|d\nu&=\int\sum_{i=0}^{R(x)-1}\log|f'(f^ix)|d\nu(x)
=\sum_{\omega\in\mathcal
Q}\sum_{i=0}^{R(\omega)-1}\int_{\omega}\log|f'(f^ix)|d\nu(x)\\
&=\sum_{\omega\in\mathcal Q}\sum_{i=0}^{R(\omega)-1}\int \log |f'|
d((f_*)^i\nu|\omega)=\int Rd\nu\int\log|f'|d\mu<\infty.
\end{align*}
The desired result follows.
\end{proof}

\section{Entropy formula}
In this last section we prove an entropy formula, connecting the
metric entropy to the Lyapunov exponent. Although this formula is
known to hold for a broad class of maps with critical and singular
points, circle maps with logarithmic singularities have not been
treated.

\begin{theorem}
$h(f,\mu)=\int\log|f'|d\mu,$ where $h(f,\mu)$ denotes the metric
entropy.
\end{theorem}
A proof of this theorem uses M\~an\'e's argument \cite{M} that is
outlined as follows. Define a family
$\{\rho_\beta\}_{\beta\in(0,\delta)}$ of functions on $S^1$ by:
$\rho_\beta(x)=d_{C}(x)$ if $d_C(x)\leq\beta$,
$\rho_\beta(x)=d_{S}(x)$ if $d_S(x)\leq\beta$, and
$\rho_\beta(x)=\beta$ in all other cases. Obviously
$0<\rho_\beta\leq\beta$ holds, and Lemma \ref{derivative} and
Lemma \ref{integrable} give $\int -\log\rho_\beta d\mu<\infty$.
Let $B(x,\rho_\beta;n):=\{y\in S^1\colon
|f^ix-f^iy|<\rho_\beta(f^ix) \ 0\leq\forall i<n\}$.
From \cite{M}, for $\mu$-a.e. $x\in S^1$,
$$\sup_{\beta>0}\limsup_{n\to\infty}
-\frac{1}{n}\log\mu(B(x,\rho_\beta;n))=h(f,\mu).$$ Fix $\kappa>0$
so that $Z=\{y\in S^1\colon \frac{d\mu}{d{\rm Leb}}(y)\geq
\kappa\}$ has positive Lebesgue measure. We show that, for any
$\beta$ and a.e. $x\in Z$, the $\limsup$ converges to
$\bar\lambda=\int\log|f'|d\mu$.

\subsection{A lemma}
For $x\in S^1$ and $n>0$, let $J_n(x)=[x-D_n(x),x+D_n(x)]$. Let
$\beta\cdot J_n(x)$ denote the interval of length $2\beta D_n(x)$
centered at $x$.
\begin{lemma}\label{incl}
For any $x\in S^1$ and $n>0$ we have:
\smallskip

\noindent $({\rm a})$ $B(x,\rho;n)\supset \beta\cdot J_n(x)$;

\noindent $({\rm b})$ if $n$ is a deep return time of $x$, then
$B(x,\rho;n+1)\subset J_n(x)$.
\end{lemma}
\begin{proof}
Lemma \ref{dist} gives $|f^i(x)-f^i(y)|\leq
\frac{2\beta}{\sqrt{L}}d_C(f^i(x))d_S(f^i(x))<\beta$ for all $y\in
\beta\cdot J_n(x)$ and every $0\leq i<n$. If
$\rho_{\beta}(f^i(x))<\beta$, then $|f^i(x)-f^i(y)|
<\beta(f^i(x)).$ This proves (a). (b) follows from Lemma \ref{exp}
and the bounded distortion of $f^n|J_n(x)$ from Lemma \ref{dist}.
\end{proof}

\subsection{Upper estimate} From the ergodic theorem,
there exists a sequence $(\theta_k)_k$ of positive numbers such
that $\theta_k\to0$ and the following holds for each $\theta_k$:
for any small $\epsilon>0$ and $\mu$-a.e. $x$, there exists $n(x)$
such that for all $n\geq n(x)$, $\min\{d_{C}(f^nx),d_S(f^nx)\}\geq
e^{-(\theta_k+\epsilon) n}$. Also, for any small $\epsilon>0$ and
$\mu$-a.e. $x$ there exists $n'(x)$ such that for all $n\geq
n'(x)$, $e^{(\bar\lambda-\epsilon)n}\leq|(f^n)'x|\leq
e^{(\bar\lambda+\epsilon)n}$. In view of the definition
(\ref{Theta}), for $\mu$-a.e. $x$ and all large $n$ depending on
$x$, \begin{equation}\label{In}|(J_n(x))|\geq
\frac{1}{\sqrt{L}}e^{-(\bar\lambda+3\epsilon)n-2\theta_k
n}.\end{equation}

Choose $x$ to be a Lebesgue density point of $Z$. As
$|J_n(x)|\to0$,
 $\mu(J_n(x))\geq \kappa|J_n(x)|$ holds for all large $n$.
 (a) in Lemma \ref{incl} and (\ref{In}) give
$$-\frac{1}{n}\log\mu(B(x,\rho;n))\leq-\frac{1}{n}\log\mu(I_n(x))
\leq \bar\lambda+4\epsilon+2\theta_k.$$ Hence
$\limsup_{n\to\infty}
-\frac{1}{n}\log\mu(B(x,\rho_\beta;n))\leq\bar\lambda
+4\epsilon+2\theta_k$ holds for $\mu$-a.e. $x\in Z$. As $\theta_k$
and $\epsilon$ can be made arbitrarily small, the desired upper
estimate holds.

\subsection{Lower estimate} Write $J$ for $J_n(x)$.
For $y\in S^1$, let $A(y)=\sharp\{0\leq i<R(y)\colon f^i(y)\in
J\}.$ Let $N>0$  be a large integer. Let $B(y)=\sharp\{0\leq
i<N\colon F^i(y)\in J\}.$ Let $T(y)=\sum_{i=0}^{N-1}A(F^iy)$ and
$U(y)=\sum_{i=0}^{N-1}R(F^iy).$ Clearly, $T(y)\leq U(y)$ and
$T(y)\leq \max_{0\leq i<N}\{R(F^iy)\}B(y)$ hold.

Let $m=\frac{30\bar\lambda n}{\lambda{\rm ln}L}.$ Let $Y=\{y\in
S^1\colon U(y)\leq m\}.$
 By the $F$-invariance of $\nu$,
\begin{equation}\label{sous} N\int Ad\nu=\int T d\nu=
\int_{Y} T d\nu+\int_{Y^c} T d\nu.
\end{equation}
We estimate the first integral of the right-hand-side.  Let
$X=\{y\colon |B(y)/N-\nu(J)|<\nu(J)\}.$ We have $\int_{X\cap Y} S
d\nu\leq \int_{X} mBd\nu$, because $T\leq mB$ holds on $Y$. The
definition of $X$ gives $\int_{X} mBd\nu\leq KmN\nu(J)$, where $K$
is a uniform constant bounding the density of $\nu$. Hence
\begin{equation}\label{sous1}
\int_YTd\nu=\int_{X\cap Y} T d\nu+ \int_{X^c\cap Y} T d\nu \leq
KmN\nu(J)+m\nu(X^c).\end{equation} The second term of the
light-hand-side can be made arbitrarily small by making $N$ large.
Indeed, by the ergodic theorem for $(F,\nu)$,
 $B/N\to\nu(J)$ a.e. as $N\to\infty$.
 The convergence in probability
gives $\nu(X^c)\to0$ as $N\to\infty$.

We now estimate the second integral of the right-hand-side of
(\ref{sous}). For a given $N$-string $(a_1,\cdots,a_{N})$ of
positive integers, let $R_{a_1\cdots a_N}=\left\{y\in S^1\colon
R(F^iy)=a_i,1\leq i\leq N\right\}.$ For each component $Q$ of
$R_{a_1\cdots a_N}$, (c) in Proposition \ref{Mar} and the bounded
distortion of $F^{a_1+\cdots+a_{N-1}}|Q$ give
$|\{ y\in Q\colon R(F^{a_1+\cdots+a_{N-1}}(y))=a_{N}\}|\leq
2\delta L^{-\frac{\lambda a_{N}}{26}}|Q|.$ Summing this over all
components gives $|R_{a_1\cdots a_{N}}|\leq L^{-\frac{\lambda
a_{N}}{26}}|R_{a_1\cdots a_{N-1}}|,$ and therefore $|R_{a_1\cdots
a_N}|\leq L^{-\frac{\lambda}{26} (a_1+\cdots+a_N)}$. This yields
\begin{equation*}\int_{Y^c} T d\nu\leq \int_{Y^c} U d\nu
=\sum_{\stackrel{(a_1,\cdots,a_N)}{a_1+\cdots+a_N>m
}}(a_1+\cdots+a_N)\nu(R_{a_1\cdots a_N}) \leq
K\sum_{r>m}\sum_{\stackrel{(a_1,\cdots,a_N)}{a_1+\cdots+a_N=r
}}L^{-\frac{\lambda r}{27}}.\end{equation*} As $R>M_0$,
$\frac{N}{r}\leq\frac{1}{M_0}$ holds. By Stirling's formula for
factorials, the number of all feasible $(a_1,\cdots,a_N)$ with
$a_1+\cdots+a_N=r$ is $\leq e^{cr}$, where $c\to0$ as
$L\to\infty$. Hence
\begin{equation}\label{sous3}\int_{Y^c} T d\nu \leq
KL^{-\frac{\lambda m }{28}}\leq K|J|.\end{equation} The last
inequality follows from (\ref{In}) and the definition of $m$.

Plugging (\ref{sous1}) (\ref{sous3}) into the right-hand-side of
(\ref{sous}), dividing the result by $N$ and passing $N\to\infty$
we obtain
\begin{equation}\label{pass}
\mu(J_n(x))\int Rd\nu=\int A d\nu\leq Km\nu(J_n(x)) =\frac{30K\bar
\lambda}{\lambda{\rm ln} L}n|J_n(x)|.\end{equation} The equality
is from (\ref{mu}).

Choose  $(n_k)_k$ denote an increasing infinite sequence of deep
return times of $x$. By the Poincar\'e recurrence, $\mu$-a.e.
point in $Z$ has such a sequence. Let $\epsilon>0$ be an arbitrary
small number. For any sufficiently large $n_k$, $|J_{n_k}(x)|\leq
e^{-(\bar\lambda-\epsilon)n_k}$ holds. For such $n_k$,
(\ref{pass}) and (b) in Lemma \ref{incl} give
$$-\frac{1}{n_k+1}\log\mu(B(x,\rho;n_k+1))\geq
-\frac{1}{n_k+1}\log\mu(J_{n_k}(x))\geq -\frac{\log
n_k}{2(n_k+1)}+\frac{n_k}{n_k+1}(\bar\lambda-\epsilon).$$ Taking
the limit $k\to\infty$ gives $\limsup_{n\to\infty}-\frac{1}{n}
\log\mu(B(x,\rho;n))\geq\bar\lambda-\epsilon$. As $\epsilon$ is
arbitrary, the desired lower estimate holds. This completes the
proof of Theorem 5.1.

\end{document}